\documentclass[12pt]{amsart}

\usepackage{fullpage}
\usepackage{amsmath}
\usepackage{amsfonts}
\usepackage{amssymb}
\usepackage{amsthm}

\usepackage{color}
\usepackage{xcolor}

\newtheorem{theorem}{Theorem}[section]
\newtheorem{lemma}[theorem]{Lemma}
\newtheorem{corollary}[theorem]{Corollary}
\newtheorem{proposition}[theorem]{Proposition}

\theoremstyle{definition}
\newtheorem{definition}[theorem]{Definition}

\theoremstyle{remark}
\newtheorem{remark}[theorem]{Remark}

\numberwithin{equation}{section}

\begin{document}

\title{Flat structure of meromorphic connections on Riemann surfaces}

\author{Karim Rakhimov}

\address{Department of  Mathematics, University of Pisa, Pisa, Italy, 56127}

\curraddr{V.I. Romanovskiy Institute of Mathematics of Uzbek Academy of Sciences,  Tashkent, Uzbekistan}
\email{karimjon1705@gmail.com}

\date{}


\begin{abstract}
The possible omega limit sets of simple geodesics for meromorphic connections on compact Riemann surfaces have been studied by Abate, Tovena and Bianchi. In this paper we study the same problem for infinite self-intersecting geodesics.  In the first part of the paper we study relation among meromorphic $k$-differentials, singular flat metrics and meromorphic connections. 
Moreover, we prove  a Poincar\'e-Bendixson theorem for infinite self-intersecting geodesics of meromorphic connections with monodromy in $G$, where $\arg G^k=\{0\}$ for some $k\in\mathbb{N}$.
\end{abstract}

\maketitle
\section{Introduction}

 Meromorphic connections on Riemann surfaces have been well studied by many authors (see for example \cite{IYa}). A \textit{meromorphic connection} on a Riemann surface $S$ is a $\mathbb{C}-$linear operator $\nabla:\mathcal{M}_{TS}\to \mathcal{M}_S^1 \otimes\mathcal{M}_{TS}$, where $\mathcal{M}_{TS}$ is the sheaf of germs of meromorphic sections of the tangent bundle $TS$ and $\mathcal{M}_S^1$ is the space of meromorphic 1-forms on $S$, satisfying the Leibniz rule
$\nabla(fs)=\text{d}f\otimes s +f\nabla s$
for all $s\in \mathcal{M}_{TS}$ and $f\in \mathcal{M}_{S}$. A \textit{geodesic} for a meromorphic connection $\nabla$ is a real smooth curve $\sigma:I\to S^o$, where $I\subseteq\mathbb{R}$ is an interval and $S^o$ is the complement of the poles of $\nabla$ in $S$, satisfying the geodesic equation $\nabla_{\sigma'}\sigma'\equiv0$. To the best of our knowledge, geodesics for meromorphic connections in this sense were first introduced in \cite{AT1}. In \cite{AT1}, the authors discovered that there is a strong relationship
between the local dynamics of the time-one map of homogeneous vector fields and the dynamics of geodesics for meromorphic connections on Riemann surfaces.
To describe this relationship,
we need to introduce a few notations and definitions.

Let $Q$ be a homogeneous vector field on $\mathbb{C}^n$ of degree $\nu+1\ge2$. First of all, notice that
the orbits of its time-1 map are contained in the real integral curves of $Q$; so we are interested
in studying the dynamics of the real integral curves of the complex homogeneous vector
field $Q$.

A \emph{characteristic direction} for $Q$ is a direction $v\in \mathbb{P}^{n-1}(\mathbb{C})$ such that the complex line
(the characteristic leaf) issuing from the origin in the direction $v$ is $Q$-invariant. An integral
curve issuing from a point of a characteristic leaf stays in that leaf forever; so the dynamics
in a characteristic leaf is one-dimensional, and thus completely known. In particular, if the
vector field $Q$ is a multiple of the radial field (we shall say that $Q$ is \emph{dicritical}) every direction
is characteristic; thus the dynamics is one-dimensional and completely understood. So, we are
mainly interested in understanding the dynamics of integral curves outside the characteristic
leaves of non-dicritical vector fields. In \cite{AT1} Abate and Tovena proved the following result:
\begin{theorem}[{Abate and Tovena \cite{AT1}}]\label{abatetovenathqtom}
  Let $Q$ be a non-dicritical homogeneous vector field of degree $\nu+1\ge 2$ in $\mathbb{C}^n$ and let $M_Q$ be the complement in $\mathbb{C}^n$ of the characteristic leaves of $Q$. Let $[\cdot]:\mathbb{C}^n\setminus\{O\}\to\mathbb{P}^{n-1}(\mathbb{C})$ denote the canonical projection. Then there exists a singular holomorphic foliation $\mathcal{F}$ of $\mathbb{P}^{n-1}(\mathbb{C})$ in Riemann surfaces, and a partial meromorphic
connection $\nabla$ inducing a meromorphic connection on each leaf of $\mathcal{F}$, whose poles coincide
with the characteristic directions of $Q$, such that the following holds:
\begin{itemize}
  \item if $\gamma: I \to M_Q$ is an integral curve of $Q$ then the image of $[\gamma]$ is contained in a leaf $S$
of $\mathcal{F}$ and it is a geodesic for $\nabla$ in $S$;
  \item conversely, if $\sigma : I\to \mathbb{P}^{n-1}(\mathbb{C})$ is a geodesic for $\nabla$ in a leaf $S$ of $\mathcal{F}$ then there are
exactly $\nu$ integral curves $\gamma_1,...,\gamma_\nu : I \to M_Q$ such that $[\gamma_j] = \sigma$ for $j = 1,...,\nu.$
\end{itemize}

\end{theorem}

Due to this result, we see that the study of integral curves for a homogeneous vector
field in $\mathbb{C}^n$ is reduced to the study of geodesics for meromorphic connections on a Riemann surface $S$.

In \cite{AT1}, Abate and Tovena studied the $\omega$-limit sets of the geodesics of meromorphic connections on $\mathbb{P}^1(\mathbb{C})$. They gave complete classification of the $\omega$-limit sets of simple geodesics. Later, in \cite{AB}, Abate and Bianchi studied the same problem for any compact Riemann surface $S$ and they proved Poincaré-Bendixson theorems for simple geodesics. These works, however, left open the following question: ``what happens if a geodesic intersects itself infinitely many times?". For simplicity up to now if there is no confusion when we say ``self-intersecting geodesics" we mean ``infinite self-intersecting geodesics".

In the first part of this work we study the relation between meromorphic connections and meromorphic $k$-differentials. In the second part of this work we study self-intersecting geodesics of a class of meromorphic connections. Before stating the main results of the paper we need to introduce a few notations and definitions.

Let $\nabla$ be a meromorphic connection on a Riemann surface $S$.
Let $\{(U_\alpha,z_\alpha)\}$ be an atlas for $S$. It is not difficult to see that there exists $\eta_\alpha\in \mathcal{M}_S^1(U_\alpha)$ on $U_\alpha$ such that $\nabla\left(\frac{\partial}{\partial z_\alpha}\right)=\eta_\alpha\otimes \frac{\partial}{\partial z_\alpha}$, where $\frac{\partial}{\partial z_\alpha}$ is the induced local generator of $TS$ over $U_\alpha$. We shall say $\eta_\alpha$ is the \emph{local representation} of $\nabla$ on $U_\alpha$. The \emph{residue} $\mathrm{Res}_p\nabla$ of $\nabla$ at a point $p\in S$ is the residue of any 1-form $\eta_\alpha$ representing $\nabla$ on a local chart $(U_\alpha,z_\alpha)$ at $p$. A pole $p$  is said to be \emph{Fuchsian} if the local representation  of $\nabla$ on one (and hence any) chart $(U,z)$ around $p$ has a simple pole at $p$. If all poles of $\nabla$ are Fuchsian then we shall say $\nabla$ is \emph{Fuchsian}. Let $S^o$ be the complement of the poles of $\nabla$ and $G$ a multiplicative subgroup of $\mathbb{C}^*$.  We say $\nabla$ has \emph{monodromy in} $G$ if there exists an atlas $\{(U_\alpha,z_\alpha)\}$ for $S^o$ such that the representations of $\nabla$ on $U_\alpha$'s are identically zero and the transition functions are of the form $z_\beta= a_{\alpha\beta} z_\alpha+c_{\alpha\beta}$ on $U_\alpha\cap U_\beta$, with $a_{\alpha\beta}\in G$ and $c_{\alpha\beta}\in\mathbb{C}$. We say $\nabla$ has \emph{real periods} if $G\subseteq S^1$.

To state the first main result we need to introduce the notion of meromorphic $k$-differentials. A \emph{meromorphic $k$-differential} $\phi$ on a Riemann surface $S$ is a meromorphic section of the $k$-th power of the canonical line bundle. The zeroes
and the poles of $\phi$ constitute the set $\Sigma$ of \emph{critical points} of $\phi$. It is not difficult to see that a $k$-differential $\phi$ is given locally as $\phi=\phi(z)dz^k$ (for more details see Definition \ref{kdiftarifi}). Then there is  a metric $g$ locally given as $g^{\frac{1}{2}}=|\phi(z)|^{\frac{1}{k}}|dz|$. In particular, $g$ is a flat metric on $S^o:=S\setminus\Sigma$ (see Proposition \ref{qdsfm}). Finally, a smooth curve $\sigma:[0,\varepsilon)\to S^o$ is a \emph{geodesic} for $\phi$ if it is a geodesic for $g$. When $k=1$  we get the meromorphic Abelian differential which have been thoroughly studied both from a geometrical and a dynamical point of view. Theory of translation surfaces (corresponding to Abelian differentials) provides new insights in the study of the dynamics of billiards through the methods
of algebraic geometry and renormalization theory. When $k=2$ we get the meromorphic quadratic differential  which are a well studied subject in Teichm{\"u}ller theory. Extracting the $k$-th root, one can think of an arbitrary $k$-differential as a multi-valued meromorphic Abelian differential on $S$. In general, differentials of order $k>2$ are much less studied than their quadratic counterpart.

 We say a meromorphic $k$-differential $\phi$ and a meromorphic connection $\nabla$ are \emph{adapted} to each other if they have the same geodesics and the same critical points.  Our first main theorem describes an exact relation between meromorphic $k$-differentials and meromorphic connections.
\begin{theorem}\label{kvamermain}
  Let $\nabla$ be a Fuchsian meromorphic connection on a Riemann surface $S$. If $\nabla$ has monodromy in $\mathbb{Z}_k$  and residues in $\frac{1}{k}\mathbb{Z}$ then there is a meromorphic $k$-differential $\phi$ adapted to $\nabla$  (here we identify $\mathbb{Z}_k$ with the multiplicative group of $k$-th roots of unity). Moreover, $\phi$ is  unique up to multiplication by a nonzero constant factor.

  On the other hand, if $\phi$ is a meromorphic $k$-differential on a Riemann surface $S$ then there exists a unique meromorphic connection $\nabla$ adapted to $\phi$. Moreover, $\nabla$ is Fuchsian  and it has monodromy in $\mathbb{Z}_k$ and residues in $\frac{1}{k}\mathbb{Z}$.
\end{theorem}

  In Section \ref{singularflatm} we study relation between meromorphic connections and singular flat metrics. In particular, we show that any Fuchsian meromorphic connection with real residues and real periods  counterpart to a  singular flat metric, and vice versa.  The proof of Theorem \ref{abatetovenathqtom} suggests us to study also meromorphic connections with pure complex residues. Since the monodromy around a pole is contained in the monodromy of the structure a meromorphic connection with pure complex residue has no real periods, i.e. monodromy in $G\not\subset S^1$.  The next main result (see  Subsection \ref{secpoincarebend} for the terminology) is a classifications of the $\omega$-limit sets of self-intersecting geodesics of meromorphic connections having monodromy in $G$, with $\arg G^k=\{0\}$, and for $A\subset \mathbb{C}^*$, we denote $\arg A:=\{\arg a:a\in A\}$.

\begin{theorem}\label{abateconiK}

Let $\nabla$ be a meromorphic connection on a compact Riemann surface $S$ and $\Sigma$ the set of poles of $\nabla$. Set $S^o=S\setminus\Sigma$.  Let  $\sigma:[0,\varepsilon)\to S^o$ be a maximal geodesic of $\nabla$. Assume $\nabla$ has monodromy in $G$ with $\arg G^k=\{0\}$ for some $k\in \mathbb{N}$. If $\sigma$ is a self-intersecting geodesic, then one of the following propositions holds:

\begin{enumerate}

  \item the $\omega$-limit set of $\sigma$ in $S$ is given by the support of a (possibly non-simple) closed geodesic;
  \item the $\omega$-limit set of $\sigma$ in $S$ is a graph of (possibly self-intersecting) saddle connections;

  \item  the $\omega$-limit set of $\sigma$  has non-empty interior and non-empty boundary, and each component of its boundary is a graph of geodesics such that every geodesic segment is a part of a saddle connection;
   \item the $\omega$-limit set of $\sigma$ in $S$ is all of $S$.
\end{enumerate}

\end{theorem}

\noindent It is not difficult to construct a meromorphic connection $\nabla$ having  monodromy in $G$ for some $G$ with $\arg G^k=\{0\}$ such that the $\omega$-limit set a self-intersecting geodesic of $\nabla$ is the whole surface. For example, thanks to Theorem \ref{kvamermain} it is enough to take a meromorphic 3-differential $\phi$ in double torus having at least one pole $p$ order $2$ and having no pole with the order greater than $2$. Any geodesic which is the $\omega$-limit set is the whole surface intersects itself infinitely many times around $p$.  By using Theorem 25.1 in \cite{SK} for translation surface $\tilde{\phi}$ induced by the canonical cover (see \cite[Section 2.1]{BCG}) we can see that there exists a geodesic which is the $\omega$-limit set is the whole surface. Thanks to Theorem \ref{kvamermain} and examples in \cite{NST} we have examples for all other possible cases.

The paper is organized as follows. In Section \ref{holomorphic connection} we repeat some preliminary notions. Moreover, we introduce the notion of $\nabla$-chart and $\nabla$-atlas and we prove some useful properties.   In Section \ref{singularflatm},  we introduce the notion of singular flat metrics and we
present general properties of singular flat metrics. Moreover, we study  relation between singular flat metrics and meromorphic connections. In Section \ref{chapter3},
we recall some definitions and fundamental results on the theory of meromorphic $k$-differentials. Moreover, we describe a relation between meromorphic connections and meromorphic $k$-differentials, and we prove Theorem \ref{kvamermain}.  
Finally, in Section \ref{canonicalcovering}, we introduce the notion of canonical covering for a meromorphic connection on a compact Riemann surface $S$ and we prove Theorem \ref{abateconiK}.

\subsection*{Acknowledgements} This work is part of the my PhD thesis which was prepared at the University of Pisa. The author would like to thank his advisor Marco Abate for suggesting this problem and for his patient guidance. 
The author also would like to thank Guillaume Tahar for discussions about $k$-differentials and different examples on self-intersecting geodesics.  

\section[Preliminary notions]{Preliminary notions} \label{holomorphic connection}

In this section we repeat some definitions and theorems from \cite{AT1} and \cite{AB}.

\begin{definition}
  Let $TS$ be the tangent bundle of a Riemann surface $S$. A holomorphic connection on $TS$ is a $\mathbb{C}-$linear map $\nabla:\mathcal{TS}\to \Omega_S^1\otimes\mathcal{TS}$ satisfying the Leibniz rule
  $$\nabla(se)=ds\otimes e +s\nabla e$$
  for all $s\in \mathcal{O}_S$ and $e\in\mathcal{TS}$, where $\mathcal{TS}$ denotes the sheaf of germs of holomorphic sections of $TS$, while $\mathcal{O}_{S}$ is the sheaf of germs holomorphic functions on $S$ and  $\Omega_S^1$ is the sheaf of germs of holomorphic 1-forms on $S$.
\end{definition}

Let us see  what this condition means in local coordinates.   Let $(U_\alpha,z_\alpha)$ be a local chart for $S$, and $\nabla$ a holomorphic connection on $S$.
 It is not difficult to see that there exists a holomorphic $1-$form $\eta_\alpha\in\Omega_S^1(U_\alpha)$ on  $U_\alpha$ such that
$$\nabla(\partial_\alpha)=\eta_\alpha\otimes\partial_\alpha,$$
where $\partial_\alpha:=\frac{\partial}{\partial z_\alpha}$ is the induced local generator of $TS$ over $U_\alpha$.
\begin{definition}
We say  that $\eta_\alpha$ is the \emph{local representation} of $\nabla$ on $U_\alpha$.
\end{definition}

Now,  for given local representations $\{\eta_\alpha\}$  we look for a condition which guarantees  the existence of a holomorphic connection $\nabla$.  Let $\{\xi_{\alpha\beta}\}$ be the cocycle
representing the cohomology class $\xi\in H^1(S,\mathcal{O}^*)$ of $TS$ (hence $\xi_{\alpha\beta}=\frac{\partial z_\alpha}{\partial z_\beta}$); over $U_\alpha\cap U_\beta$ we have
$$\partial_\beta=\partial_\alpha\xi_{\alpha\beta}$$
 and thus
$$\nabla(\partial_\beta)=\nabla(\partial_\alpha\xi_{\alpha\beta})\Leftrightarrow \eta_\beta\otimes \partial_\beta=\xi_{\alpha\beta}\eta_\alpha\otimes \partial_\alpha+d \xi_{\alpha\beta}\otimes \partial_\alpha$$
if and only if
\begin{equation}\label{mer1}
  \eta_\beta=\eta_\alpha+\frac{d\xi_{\alpha\beta}}{\xi_{\alpha\beta}}.
\end{equation}

Recalling the short exact sequence of sheaves
$$0\to\mathbb{C}^*\to \mathcal{O}^*\xrightarrow{\partial \log} \Omega_S^1\to 0$$
we can see that equality \eqref{mer1} shows that the existence of a holomorphic connection $\nabla$ is equivalent to the vanishing of the image of $\xi$ under the map $\partial\log: H^1(S,\mathcal{O}^*)\to H^1(S,\Omega_S^1)$ induced on cohomology. Hence the class $\xi$ is the image of a class $\hat{\xi}\in H^1(S,\mathbb{C}^*)$. We recall how to find a representative $\hat{\xi}_{\alpha\beta}$ of $\hat{\xi}.$ After shrinking the $U_\alpha$'s, if necessary,  we can find holomorphic functions $K_\alpha\in\mathcal{O}(U_\alpha)$ such that $\eta_\alpha=\partial K_\alpha$ on $U_\alpha$. Set
\begin{equation}\label{mer2}
  \hat{\xi}_{\alpha\beta}=\frac{\exp(K_\alpha)}{\exp(K_\beta)}\xi_{\alpha\beta}.
\end{equation}
in $U_\alpha\cap U_\beta$. Then \eqref{mer1} implies that $\hat{\xi}_{\alpha\beta}$
is a complex non-zero constant defining a cocycle representing $\xi$.

 \begin{definition}
   The homomorphism $\rho:\pi_1(S)\to\mathbb{C}^*$ corresponding to the class $\hat{\xi}$ under the canonical isomorphism $$H^1(S,\mathbb{C}^*)\cong\text{Hom}(H_1(S,\mathbb{Z}),\mathbb{C}^*)=\text{ Hom}(\pi_1(S),\mathbb{C}^*)$$ is the \emph{monodromy representation} of the holomorphic connection $\nabla$. We shall say that $\nabla$ has \emph{monodromy in} $G$, a multiplicative subgroup of $\mathbb{C}^*$, if the image of $\rho$ is contained in $G$, that is if $\hat{\xi}$ is the image of a class in $H^1(S,G)$ under the natural inclusion $G\hookrightarrow \mathbb{C}^*$. We say $G$ is a \emph{monodromy group} of $\nabla$. We shall say that $\nabla$ has  \emph{real periods} if it has monodromy in  $S^1$.
 \end{definition}

\begin{remark}
  In \cite{AT1} is introduced a \emph{period map} $\rho_0:H_1(S,\mathbb{Z})\to\mathbb{C}$ associated to $\nabla$ and having the following relation
  \begin{equation}\label{rhorho0}
    \rho=\exp(2\pi i\rho_0)
  \end{equation}
  with the monodromy representation $\rho$. By \eqref{rhorho0} we can see that  $\nabla$ has real periods if and only if the image of the period map is contained in $\mathbb{R}$.
\end{remark}

\subsection{$\nabla$-atlas}\label{nablaatlassect}
Let $S$ be a Riemann surface and $\nabla$ a holomorphic connection on the tangent bundle $TS$. For simplicity we just say $\nabla$ is a holomorphic connection on $S$ instead of $TS$. Let us introduce the notion of $\nabla$-chart.
 \begin{definition}\label{nablaatlasdef} Let $\nabla$ be a holomorphic connection on a Riemann surface $S$.
   A simply connected chart $(U_\alpha,z_\alpha)$  is said to be a $\nabla$-\emph{chart}, if the representation of $\nabla$ is identically zero on $U_\alpha$.
   A  \emph{$\nabla$-atlas} is an atlas $\{(U_\alpha,z_\alpha)\}$ for $S$ such that all charts are $\nabla$-charts.
 \end{definition}

 The next lemma states an equivalence between holomorphic connections and $\nabla$-atlases.
 \begin{lemma}\label{nablachart}
   Let $\nabla$ be a holomorphic connection on a Riemann surface $S$. If $p\in S$, then there exists a $\nabla$-chart around $p$. In particular, there exists a $\nabla$-atlas for $S$.
 \end{lemma}

\begin{proof}
  Let $(U_\alpha,z_\alpha)$ be a simply connected chart centered at $p$, and $\eta_\alpha$ the representation of $\nabla$ on $U_\alpha$. Let $K_\alpha$ be a holomorphic primitive of $\eta_\alpha$. Let $J_\alpha$ be a holomorphic primitive of $\exp(K_\alpha)$. Then there exists a simply connected domain $\tilde{U}_\alpha\subseteq U_\alpha$ such that $J_\alpha\colon\tilde{U}_\alpha\to J_\alpha(\tilde{U}_\alpha)$ is one-to-one.  By setting $w_\alpha:=J_\alpha$ we define a chart $(\tilde{U}_\alpha,w_\alpha)$. Let $\tilde{\eta}_\alpha$ be the representation of $\nabla$ on $\tilde{U}_\alpha$. By the transformation rule we have
  $$\eta_\alpha=\tilde{\eta}_\alpha+\frac{d\xi}{\xi}$$
  on $\tilde{U}_\alpha$, where $\xi:=\frac{dw_\alpha}{dz_\alpha}=\exp(K_\alpha)$. Consequently, $\tilde{\eta}_\alpha\equiv0$.  Hence $(\tilde{U}_\alpha,w_\alpha)$ is a $\nabla$-chart.

   As we have seen above, it is always possible to find a $\nabla$-chart around any point of $S$. Hence there exists an atlas for $S$ such that all charts are $\nabla$-charts.

\end{proof}
Let us recall the notion of Leray atlas.
\begin{definition}
    Let $S$ be a Riemann surface. A \emph{Leray atlas} for $S$  is a simply connected atlas $\{(U_\alpha,z_\alpha)\}$, such that intersection of any two charts of the atlas is  simply connected or empty.
\end{definition}
Let us study relation between monodromy group and $\nabla$-atlas
\begin{lemma}\label{nablaatlas}
  Let $G$ be  a multiplicative subgroup of $\mathbb{C}^*$.
  Let $\nabla$ be a holomorphic connection on a Riemann surface $S$ such that $\nabla$ has monodromy in $G$. Then there exists a Leray $\nabla$-atlas $\{(U_\alpha,z_\alpha)\}$ for $S$ such that the atlas has the transition functions $z_\beta= a_{\alpha\beta} z_\alpha+c_{\alpha\beta}$ on $U_\alpha\cap U_\beta$, where $a_{\alpha\beta}\in G$ and $c_{\alpha\beta}\in\mathbb{C}$.

  Conversely, if $\nabla$ is a holomorphic connection such that there exists a Leray $\nabla$-atlas $\{(U_\alpha,z_\alpha)\}$ for $S$ with the transition functions $z_\beta= a_{\alpha\beta} z_\alpha+c_{\alpha\beta}$ on $U_\alpha\cap U_\beta$, where $a_{\alpha\beta}\in G$ and $c_{\alpha\beta}\in\mathbb{C}$, then $\nabla$ has monodromy in $G$.
\end{lemma}
\begin{proof}
Let $\nabla$ be a holomorphic connection on  $S$ such that $\nabla$ has monodromy in $G$. By Lemma \ref{nablachart}  there exists a $\nabla$-atlas for $S$. Take a Leray refinement of the $\nabla$-atlas. So we have a Leray $\nabla$-atlas $\{(U_\alpha,z_\alpha)\}$. Since the representations of $\nabla$ are identically zero on $U_\alpha$  we have $\xi_{\alpha\beta}:=\frac{dz_\alpha}{dz_\beta}\in\mathbb{C}^*$  on $U_\alpha\cap U_\beta$. Since $\nabla$ has monodromy in $G$ there exists constants $c_\alpha\in\mathbb{C}^*$ such that $\xi_{\alpha\beta}=\frac{c_\beta}{c_\alpha}\hat{\xi}_{\alpha\beta}$ with $\hat{\xi}_{\alpha\beta}\in G$. By setting $w_\alpha=c_\alpha z_\alpha$ we define an atlas $\{(U_\alpha,w_\alpha)\}$ for $S$. Hence $\{(U_\alpha,w_\alpha)\}$ is a Leray $\nabla$-atlas  such that $\frac{dw_\alpha}{dw_\beta}\in G$ on  $U_\alpha\cap U_\beta$. So the atlas has the transition functions $w_\beta= a_{\alpha\beta} w_\alpha+c_{\alpha\beta}$ on $U_\alpha\cap U_\beta$, where $a_{\alpha\beta}\in G$ and $c_{\alpha\beta}\in\mathbb{C}$.

Let now  $\nabla$ be a holomorphic connection such that there exists a Leray $\nabla$-atlas $\{(U_\alpha,z_\alpha)\}$ for $S$ with the transition functions $z_\beta= a_{\alpha\beta} z_\alpha+c_{\alpha\beta}$ on $U_\alpha\cap U_\beta$, where $a_{\alpha\beta}\in G$ and $c_{\alpha\beta}\in\mathbb{C}$. Since $(U_\alpha,z_\alpha)$ is a $\nabla$-chart the representation $\eta_\alpha$ of $\nabla$ on $U_\alpha$ is identically zero. Let $K_\alpha\equiv1$ be a holomorphic primitive of $\eta_\alpha$. Then $\hat{\xi}_{\alpha\beta}=a_{\alpha\beta}=\xi_{\alpha\beta}$ in \eqref{mer2}.   Consequently, $\nabla$ has monodromy in $G$.
\end{proof}
It is well known that to a Hermitian metric $g$ on the tangent bundle over a complex manifold $M$ can be associated a connection  $\nabla$ (not necessarily
holomorphic) such that $\nabla g\equiv 0$, the \emph{Chern connection} of $g$. The converse problem was also studied by Abate and Tovena for holomorphic connections (see \cite[Proposition 1.2]{AT1}): given a holomorphic connection $\nabla$, does there exist a Hermitian metric $g$ so that  $\nabla g\equiv 0$?

\begin{definition}\label{adapted}
  Let $\nabla$ be a holomorphic connection on a Riemann surface $S$. We say that a Hermitian metric $g$ on $TS$ is \emph{adapted} to $\nabla$ if $\nabla g\equiv 0$, that is if
  $$X(g(R,T))=g(\nabla_X R, T)+g(R,\nabla_{\overline{X}}T)$$
  and
  $$\overline{X}(g(R,T))=g(\nabla_{\overline{X}}R,T)+g(R,\nabla_XT)$$
  for all sections $R$, $T$ of $TS$, and every vector field $X$ on $S$.

\end{definition}

As usual, let us check the condition in local coordinates. Let $\{(U_\alpha,z_\alpha)\}$ be an atlas for $S$. A Hermitian metric $g$ on $TS$ is locally represented by a positive $C^\infty$ function $n_\alpha\in C^\infty(U_\alpha,\mathbb{R}^+)$ given by
$$n_\alpha=g(\partial_\alpha,\partial_\alpha).$$ Then we can see that $\nabla g\equiv 0$ over $U_\alpha$ if and only if
\begin{equation}\label{metric1}
  \partial n_\alpha=n_\alpha\eta_\alpha,
\end{equation}
where $\eta_\alpha$ is the holomorphic 1-form representing $\nabla$.

\begin{proposition}[{\cite[Proposition 1.1]{AT1}}]\label{localmetric}
Let  $\nabla$ be a holomorphic connection on a Riemann surface $S$. Let $(U_\alpha,z_\alpha)$ be a local chart, and define $\eta_\alpha\in \Omega_S^1(U_\alpha)$ by setting $\nabla \partial_\alpha=\eta_\alpha\otimes \partial_\alpha$.
 Assume that we have a holomorphic primitive $K_\alpha$ of $\eta_\alpha$ on $U_\alpha$. Then
\begin{equation}\label{metric2}
  n_\alpha=\exp(2\mathrm{Re}\,K_\alpha)
\end{equation}
is a positive solution of \eqref{metric1}. Conversely, if $n_\alpha$ is a positive solution of \eqref{metric1} then for any $z_0\in U_\alpha$ and any simply
connected neighborhood $U\subseteq U_\alpha$ of $z_0$ there is a holomorphic primitive $K_\alpha\in\mathcal{O}(U)$ of $\eta_\alpha$ over $U$ such that
$n_\alpha =\exp(2 \mathrm{Re} K_\alpha)$ in $U$. Furthermore, $K_\alpha$ is unique up to a purely imaginary additive constant. Finally, two (local) solutions of \eqref{metric1} differ (locally) by a positive multiplicative constant.

\end{proposition}
 It is not difficult to see that Gaussian curvature of the \emph{local metrics} \eqref{metric2}  is identically zero. The  proposition shows that for any holomorphic connection $\nabla$ we can always associate local flat metrics $g$ adapted to $\nabla$. A \emph{global metric} adapted to $\nabla$ might not exist.
\begin{theorem}[{\cite[Proposition 1.2]{AT1}}]\label{realperiod} Let $\nabla$ be a holomorphic connection on a Riemann surface $S$. Then there exists a flat metric adapted to $\nabla$ if and only if $\nabla$ has real periods.
\end{theorem}

\begin{definition}
  A \emph{geodesic} for a holomorphic connection $\nabla$ on $S$ is a real curve $\sigma:I\to S$, with $I\subseteq\mathbb{R}$ an interval, such that $\nabla_{\sigma'}\sigma'\equiv 0$, where $\nabla_u s:=\nabla s(u)$.
\end{definition}

Let  $\{(U_\alpha,z_\alpha)\}$ be an atlas for $S$ and  $\sigma:I\to U_\alpha$, with $I\subseteq\mathbb{R}$ an interval,  a smooth curve. Then $\sigma$ is a geodesic for a meromorphic connection $\nabla$ if and only if
\begin{equation}\label{geodesicequation}
 (z_\alpha\circ\sigma)''+(f_\alpha\circ\sigma){(z_\alpha\circ\sigma)'}^2\equiv0
\end{equation}
where $\eta_\alpha=f_\alpha dz_\alpha$ is the local representation of $\nabla$ on $U_\alpha$.

\begin{proposition}\label{t1}
  Let $\nabla$ be a holomorphic connection on a Riemann surface $S$. Let $(U_\alpha,z_\alpha)$ be a $\nabla$-chart. Let $\sigma:[0,\varepsilon)\to U_\alpha$ be a smooth curve. Then $\sigma$ is a geodesic for $\nabla$ if and only if the representation $z_\alpha\circ\sigma$ is a Euclidean segment in $z(U_\alpha)$.
\end{proposition}
\begin{proof}
    Since $(U_\alpha,z_\alpha)$ is a $\nabla$-chart the local representation of $\nabla$ on $U_\alpha$ is identically zero. Consequently, by \eqref{geodesicequation} we can see that $\sigma$ is a geodesic if and only if  $$(z_\alpha\circ\sigma)''\equiv0.$$
    Hence $z_\alpha\circ\sigma$ is a Euclidean segment.

    \end{proof}

\subsection{Meromorphic connection}\label{meromorphicconnection}
\begin{definition} A \textit{meromorphic connection} on the tangent bundle $TS$ of a Riemann surface $S$ is a $\mathbb{C}$-linear map   $\nabla:\mathcal{M}_{TS}\to \mathcal{M}_S^1 \otimes\mathcal{M}_{TS}$    satisfying the Leibniz rule
$$\nabla(\tilde{f}\tilde{s})=\text{d}\tilde{f}\otimes \tilde{s} +\tilde{f}\nabla \tilde{s}$$
for all $\tilde{s}\in \mathcal{M}_{TS}$ and $\tilde{f}\in \mathcal{M}_{S}$, where $\mathcal{M}_{TS}$ denotes the sheaf of germs of meromorphic sections of $TS$, while $\mathcal{M}_S$ is the sheaf of germs of meromorphic functions and $\mathcal{M}_S^1$ is the sheaf of meromorphic 1-forms on $S$.
\end{definition}
 Let $(U_\alpha,z_\alpha)$ be a local chart for $S$, and $\nabla$ a meromorphic connection on $S$.  Then there exists $\eta_\alpha\in \mathcal{M}_S^1(U_\alpha)$, such that
$$\nabla(\partial_\alpha)=\eta_\alpha\otimes\partial_\alpha,$$
where $\partial_\alpha:=\frac{\partial}{\partial z_\alpha}$ is the induced local generator of $TS$ over $U_\alpha$.  As in the case of holomorphic connection we have

\begin{equation}\label{merrepresentation}
  \eta_\beta=\eta_\alpha+\frac{1}{\xi_{\alpha\beta}}\partial \xi_{\alpha\beta}
\end{equation}
on $U_\alpha\cap U_\beta$, where $\xi_{\alpha\beta}:=\frac{\partial z_\alpha}{\partial z_\beta}$.
In particular, if every representation is holomorphic then $\nabla$ is a holomorphic connection. We say $p\in S$ is a \emph{pole} for a meromorphic connection $\nabla$ if $p$ is a pole of $\eta_\alpha$ for some (and hence any) local chart $U_\alpha$ at $p$. If $\Sigma$ is the set of poles of $\nabla$, then $\nabla$ is a holomorphic connection on $S^o=S\setminus\Sigma$. So we can define a notion of geodesic for $\nabla$ on $S^o$ as in holomorphic connection.
\begin{definition}
A \textit{geodesic} for a meromorphic connection $\nabla$ on $TS$ is a real curve $\sigma:I\to S^o$, with $I\subseteq \mathbb{R}$ an interval, such that $\nabla_{\sigma'}\sigma'\equiv 0.$
\end{definition}

\begin{definition} The \emph{residue} $\mathrm{Res}_p\nabla$ of a meromorphic connection $\nabla$ at a point $p\in S$ is the residue of any 1-form $\eta_\alpha$ representing $\nabla$ on a local chart $(U_\alpha,z_\alpha)$ at $p$.  The set of all residues of $\nabla$ is denoted by $\mathrm{Res}\nabla$, i.e., $\mathrm{Res}\nabla:=\{\mathrm{Res}_p\nabla: p\in S\}\setminus \{0\}$.
\end{definition}

\noindent By condition \eqref{merrepresentation} the residue of $\nabla$ does not depend on the choice of charts.
\begin{definition} We say that $p\in S$ is a \textit{Fuchsian} pole of a meromorphic connection $\nabla$ if there exists a (and hence any) chart $(U_\alpha,z_\alpha)$ around $p$ such that the representation of $\nabla$ has a simple pole at $p$. If all poles of $\nabla$ are Fuchsian then we say $\nabla$ is a \emph{Fuchsian meromorphic connection}.
\end{definition}

\subsection{Poincaré-Bendixson theorems}\label{secpoincarebend}
In this subsection we recall Poincaré-Bendixson theorems for meromorphic connections on compact Riemann surfaces, i.e., a classification of the possible $\omega$-limit sets for the geodesics of meromorphic connections on compact Riemann surfaces.
\begin{definition}

Let $\sigma:(\varepsilon_-,\varepsilon_+)\to S$ be a curve in a Riemann surface $S$. Then the \emph{$\omega$-limit set} of $\sigma$ is given by the points $p\in S$ such that there exists a sequence $\{t_n\}$, with $t_n\uparrow \varepsilon_+$, such that $\sigma(t_n)\to p$. Similarly,  the \emph{$\alpha$-limit set} of $\sigma$ is given by the points $p\in S$ such that there exists a sequence $\{t_n\}$, with $t_n\downarrow \varepsilon_-$, such that $\sigma(t_n)\to p$.

\end{definition}

\begin{definition}
  A geodesic $\sigma:[0,l]\to S$ is \emph{closed} if $\sigma(l)=\sigma(0)$ and $\sigma'(l)$ is a positive multiple of $\sigma'(0)$ ; it is\emph{ periodic} if $\sigma(l)=\sigma(0)$  and $\sigma'(l)=\sigma'(0)$.
\end{definition}

\begin{definition}
  A \emph{geodesic segment} for a meromorphic connection $\nabla$ on  $S$ is a maximal geodesic $\sigma:(\varepsilon_{-},\varepsilon_{+}) \to S^o$ (with $\varepsilon_{-}\in [-\infty,0)$ and $\varepsilon_{+}\in (0,+\infty]$) such that $\sigma(t)$ tends to a pole or to a regular point of $\nabla$ as $t\uparrow \varepsilon_{+}$ or $t\downarrow\varepsilon_{-}$,  where
$S^o:=S\setminus\{p_0,p_1,\dots,p_r\}$ and $p_0,p_1,\dots,p_r$ are the poles of $\nabla$. A \emph{saddle connection} is a geodesic segment which is the endpoints are poles.

  A \textit{graph of geodesics} is a connected  graph in $S$  whose arcs are disjoint geodesic segments (note that a vertex of a graph of geodesics can be a regular point or a pole). A \emph{graph of saddle connections} is a connected  graph in $S$ whose vertices are poles and whose arcs are disjoint saddle connections. A \emph{spike} is a saddle connection of a graph which does not belong to any cycle of the graph.

A \emph{boundary graph of saddle connections}  is a graph of saddle connections which is also the boundary of a connected open subset of $S$. A boundary graph is
\emph{disconnecting} if its complement in $S$ is not connected.

 A set $V\subset S$ with $\mathring V=\emptyset$ is called a \textit{transversally Cantor-like geodesic set} if the following  conditions holds:
  \begin{enumerate}
        \item[\rm(i)] there exists a maximal non self-intersecting geodesic $\sigma\colon(\varepsilon_-,\varepsilon_+)\to S^o$ such that $V$ is the closure of the support of~$\sigma$;
         \item[\rm(ii)] for any simple geodesic $\gamma\colon(-\delta,\delta)\to S^o$ transverse to $\sigma$ the intersection $\gamma([-\delta/2,\delta/2])\cap V$ is a perfect totally disconnected set (a Cantor set).
  \end{enumerate}

\end{definition}

Next we state the Poincaré-Bendixson theorem for a meromorphic connection on a compact Riemann surface $S$ which proved in \cite[Theorem 4.6]{AT1} and \cite[Theorem 4.3]{AB} (see also  \cite[Remark 2.19]{AR24}).
 \begin{theorem}[{Abate-Bianchi-Tovena}]\label{t3}
Let $\sigma:[0,\varepsilon)\to S^o$ be a maximal geodesic for a meromorphic connection $\nabla$ on $S$, where
$S^o=S\setminus\{p_0,p_1,\dots,p_r\}$ and $p_0,p_1,\dots,p_r$ are the poles of $\nabla$. Then one of the following propositions holds:
\begin{enumerate}

  \item $\sigma(t)$ tends to a pole of $\nabla$ as $t\to\varepsilon$;
  \item $\sigma$ is closed;
  \item the $\omega$-limit set of $\sigma$ in $S$ is given by the support of a closed geodesic;
  \item the $\omega$-limit set of $\sigma$ in $S$ is a boundary graph of saddle connections;
   \item  the $\omega$-limit set of $\sigma$ is a transversally Cantor-like geodesic set;
  \item the $\omega$-limit set of $\sigma$ in $S$ is all of $S$;
  \item the $\omega$-limit set of $\sigma$  has non-empty interior and non-empty boundary, and each component of its boundary is a graph of saddle connections with no spikes and at least one pole;
   \item  $\sigma$ is a self-intersecting geodesic.

\end{enumerate}
  Furthermore, in cases 2 or 3 the support of $\sigma$ is contained in only one of the components of the complement of the $\omega$-limit set, which is a part $P$ of $S$ having the $\omega$-limit set as boundary.

\end{theorem}

\subsection{Local behavior of Fuchsian poles}

\begin{definition}
  Let $\nabla$ be a meromorphic connection on a Riemann surface $S$ and $p_0$ a Fuchsian pole for $\nabla$. A Fuchsian pole $p_0$ is said to be \emph{resonant} if $-1-\mathrm{Res}_{p_0}\nabla\in\mathbb{N}^*$, \emph{non-resonant} othervise.
\end{definition}
To study the $\omega$-limit set of a geodesic of a meromorphic connection $\nabla$ on a Riemann surface it is useful to know  local behavior of geodesics around a pole $p_0$.
  Proposition 8.4 in \cite{AT1} gives such an understanding  around the non-resonant Fuchsian poles. Moreover by using  Lemma 3.1 in \cite{T} and Theorem \ref{kvamermain} we can see that there is no geometric difference between resonant and non-resonant poles.
\begin{theorem}[{\cite[Proposition 8.4]{AT1} and \cite[Lemma 3.1]{T}}] \label{fuchsianlocalb}
  Let $\nabla$ be a meromorphic connection on a Riemann surface $S$. Let $p_0$ be a Fuchsian pole for $\nabla$. Let $\rho:=\mathrm{Res}_{p_0}\nabla$. Then there is a neighborhood $U$ of $p_0$ such that:
  \begin{enumerate}
    \item if $\mathrm{Re}\,\rho<-1$ then every geodesic ray which enters into $U$    tends to $p_0$. The two rays of any geodesic which stays in $U$ tend to $p_0$;
    \item if $\mathrm{Re}\,\rho>-1$ then all geodesics but one issuing from any point $p\in U\setminus\{p_0\}$ escapes $U$;
    \item if $\mathrm{Re}\,\rho=-1$ but $\rho\ne-1$ then the geodesics not escaping $U$ are either closed or accumulate the support of a closed geodesic in $U$;
    \item if $\rho=-1$ then any  maximal geodesic $\sigma:I\to U\setminus\{p_0\}$, maximal in both forward and backward time, is either  periodic or escapes $U$ in one ray and tends to $p_0$ in another ray.
  \end{enumerate}

\end{theorem}
As a consequence we get
\begin{corollary}\label{rhoge-1}
Let $\nabla$ be a meromorphic connection on a Riemann surface $S$. Set $S^o:=S\setminus\Sigma$ where $\Sigma$ is the set of poles for $\nabla$.  Let $\sigma:[0,\varepsilon)\to S^o$ be a maximal geodesic of $\nabla$ and $W$  its $\omega$-limit set. Let $p$ be a Fuchsian pole with $\mathrm{Re}\,\mathrm{Res}_{p}\nabla\le -1$. If $p\in W$ then $W=\{p\}$.
\end{corollary}

\section{Singular flat metrics}\label{singularflatm}

In this section we define the notion of singular flat metric and we study some of its properties.

\begin{definition}
 Let $S$ be a Riemann surface and $\Sigma=\{p_1,...,p_r\}$ a finite set. Set $S^o:=S\setminus\Sigma$. We say that $g$ is a \emph{singular flat metric} on $S$, if $g$ is a flat metric on $S^o$ and  for any $p\in\Sigma$ there exist $c_p,b_p\in \mathbb{R}$ with $ b_p>0$ such that if $(U_\alpha,z_\alpha)$ is a chart centered $p$, with $U_\alpha\cap \Sigma=\{p\}$, then the flat metric $g^{\frac{1}{2}}=e^{u_\alpha}|dz_\alpha|$ on $U_\alpha\setminus\{p\}$ satisfies $$\lim\limits_{z_\alpha\to 0}\frac{e^{u_\alpha}}{|z_\alpha|^{c_p}}=b_p$$
  where $u_\alpha:U_\alpha\setminus\{0\}\to\mathbb{R}$ is a harmonic function. We say $p$ is a \emph{conical singularity} of \emph{order} $c_p$  of the metric $g$ and $\Sigma$ is the \emph{singular set} of $g$. If the  order $c_p$ of a conical singularity $p$ is greater then $-1$ it is usually called  conical singularity of \emph{angle} $\theta:=2\pi(c_p+1)$.

\end{definition}

\begin{remark}
 The order $c_p$ does not depend on the chosen chart. Let $(U_\beta,z_\beta)$ be another chart centered $p$ and $g^{\frac{1}{2}}=e^{u_\beta}|dz_\beta|$ on $U_\alpha\setminus\{p\}$ for some harmonic function $u_\beta:U_\alpha\setminus\{p\}\to\mathbb{R}$. Then we have
  $$e^{u_\beta}=e^{u_\alpha}\left|\xi_{\alpha\beta}\right|,$$
where $\xi_{\alpha\beta}=\frac{dz_\alpha}{dz_\beta}$. Consequently,
$$\lim\limits_{z_\beta\to 0}\frac{e^{u_\beta}}{|z_\beta|^{c_p}}=\lim\limits_{z_\beta\to 0}\frac{e^{u_\alpha}}{|z_\beta|^{c_p}}\left|\xi_{\alpha\beta}\right|=\lim\limits_{z_\beta\to 0}\frac{e^{u_\alpha}|z_\alpha|^{c_p}}{|z_\alpha|^{c_p}|z_\beta|^{c_p}}\left|\xi_{\alpha\beta}\right|=
b_p|\xi_{\alpha\beta}(p)|^{c_p+1}.$$
Since $\xi_{\alpha\beta}(p)\ne0$, we conclude that $c_p$ does not depend on the chosen chart.
\end{remark}

\begin{lemma}[{see \cite{T91}}]\label{l1}
  If $g$ is a singular flat metric on a Riemann surface $S$ and $p$ a conical singularity of order $c_p$ then for any simply connected chart $(U, z)$ centered at $p$ with $U\cap\Sigma=\{p\}$ there exists a holomorphic function $F:U\to \mathbb{C}$ such that the flat metric is given by
  $$g^{\frac{1}{2}}=|z|^{c_p}|e^{F(z)}dz|$$ on $U\setminus\{p\}.$
\end{lemma}

\subsection[Singular flat metrics \& meromorphic connections]{Relation between singular flat metrics and meromorphic connections}\label{singularflatandmeromcon}

Let $\nabla$ be a holomorphic connection on a Riemann surface $S$. As we have seen in Theorem \ref{realperiod} there exists a flat metric $g$ adapted to $\nabla$ if and only if  $\nabla$ has real periods. Note that if $g$ is adapted to $\nabla$ then they have the same geodesics.

\begin{definition} Let $\nabla$ be a meromorphic connection on a Riemann surface $S$ and let $\Sigma$ denote the set of poles of $\nabla$.  We say that a singular flat metric $g$ on $S$ is \emph{adapted} to $\nabla$ if it is adapted to $\nabla$ on $S^o:=S\setminus\Sigma$ and $\Sigma$ is the singular set of $g$.
\end{definition}
\begin{definition}
    Let $S$ be a Riemann surface and $\Sigma\subset S$ a discrete set not having limit points in $S$. Set $S^o=S\setminus\Sigma$. A \emph{Leray atlas adapted} to $(S^o,\Sigma)$ is a Leray atlas $\{(U_\alpha,z_\alpha)\}\cup\{(U_k,z_k)\}$ of $S$ such that $\{(U_\alpha,z_\alpha)\}$ is a Leray atlas for $S^o$, each $(U_k,z_k)$ is a simply connected chart centered at $p_k\in\Sigma$  and $U_k\cap U_h=\emptyset$ if $k\ne h$.
\end{definition}
Let us state an analogue of Theorem \ref{realperiod}
\begin{theorem}\label{t8}

 Let $\nabla$ be a Fuchsian meromorphic connection on a Riemann surface $S$, and $\Sigma$ the set of poles of $\nabla$. Set $S^o=S\setminus\Sigma$. If $\nabla$ has real periods on $S^o$ and $\mathrm{Res}\nabla\subset\mathbb{R}$ then there exists a singular flat metric $g$ adapted to $\nabla$. Moreover, $g$ is unique up to a positive constant multiple.

 Conversely, if $g$ is a singular flat metric on $S$ with singular set $\Sigma$ then there exists a unique meromorphic connection $\nabla$ with $\Sigma$ as set  of poles such that $g$ is adapted to $\nabla$. Moreover, $\nabla$ is Fuchsian, it has real periods on $S^o$ and $\mathrm{Res}\nabla\subset\mathbb{R}$.

 Furthermore, if $c_p$ is the order of a conical singularity $p$ of $g$ then $\mathrm{Res}_p\nabla=c_p$ and vice versa.

\end{theorem}

\begin{proof}
Let  $\nabla$ be a Fuchsian meromorphic connection on a Riemann surface $S$ with real periods and $\mathrm{Res}\nabla\subset\mathbb{R}$. Let $\Sigma$ be the set of poles of $\nabla$. Set $S\setminus\Sigma$. By Theorem \ref{realperiod} there exists a flat metric $g$ adapted to $\nabla$  on $S^o$. Let $\{(U_\alpha,z_\alpha)\}\cup\{(U_k,z_k)\}$ be a Leray atlas adapted to  $(S^o,\Sigma)$.  Let $\eta_\alpha$ be  the representation of $\nabla$ on $U_\alpha$. By Proposition \ref{localmetric}  $g$ is defined
\begin{equation}\label{singularbilanmern}
  g^{\frac12}_\alpha=\exp(\mathrm{Re}\,F_\alpha)|dz_\alpha|
\end{equation}
 on $U_\alpha$ for a suitable holomorphic primitive $F_\alpha$ of $\eta_\alpha$.
 Let $p_k$ be a pole of $\nabla$. Let $(U_k,z_k)$ be the chart centered at $p$. Let
 $$\eta_k=\left(\frac{c_k}{z_k}+f_k\right)dz_k$$
 be the representation of $\nabla$ on $U_k$, where $f_k:U_k\to \mathbb{C}$ is a holomorphic function and $c_k:=\mathrm{Res}_{p_k}\nabla$. Let $V\subset U_k\setminus\{p_k\}$ be a simply connected open set. Then by Proposition \ref{localmetric} for a suitable holomorphic primitive $K_k$ of $\eta_k$ on $V$ we have
 $$g^{\frac12}=\exp(\mathrm{Re}\,K_k)|dz_k|.$$
 Then $K_k=c_k\log z_k+ F_k$ for a holomorphic primitive $F_k$ of $f_k$ on $V$. Hence
 $$g^{\frac{1}{2}}=|z_k|^{c_k}|e^{F_k}dz_k|$$
on $V$. Since $F_k$ is a holomorphic  primitive of $f_k$ on $V$ and $f_k$ is a holomorphic function on $U_k$ there exists a holomorphic primitive $\tilde{F}_k$ of $f_k$ on $U_k$ such that $\tilde{F}_k|_V=F_k$.  Consequently, we have
$$g^{\frac{1}{2}}=|z_k|^{c_k}|e^{\tilde{F}_k}dz_k|$$
on $U_k$. Hence
$$\lim_{z_k\to 0}\frac{|z_k|^{c_k}|e^{\tilde{F}_k}|}{|z_k|^{c_k}}=e^{\tilde{F}(0)}.$$
By definition of singular flat metrics we can see that $p_k$ is  a conical singularity  for $g$ with order $c_k$. Hence $g$ has a continuation as a singular flat metric to any pole of $\nabla$. Since $g$ is adapted to $\nabla$ on $S^o$ it is adapted to $\nabla$ on $S$.

Let $\tilde{g}$ be another singular flat metric adapted to $\nabla$. Let $\eta_\alpha$ be the representation of $\nabla$ on $(U_\alpha,z_\alpha)$. For suitable holomorphic primitives $F_\alpha$ and $\tilde{F}_\alpha$ of $\eta_\alpha$ we have
$$g^{\frac12}_\alpha=\exp(\mathrm{Re}\,F_\alpha)|dz_\alpha|$$
and
$$\tilde{g}^{\frac12}_\alpha=\exp(\mathrm{Re}\,\tilde{F}_\alpha)|dz_\alpha|$$
on $U_\alpha$. Since $F_\alpha$ and $\tilde{F}_\alpha$ holomorphic primitives of $\eta_\alpha$ there exists $C_\alpha\in\mathbb{C}$ such that
$$F_\alpha=\tilde{F}_\alpha+C_\alpha.$$
Hence $g^{\frac12}_\alpha=|e^{C_\alpha}|\tilde{g}^{\frac12}_\alpha$ on $U_\alpha$. Let $(U_\beta,z_\beta)$ be a chart with $U_\alpha\cap U_\beta\ne\emptyset$. Then we have $g^{\frac12}_\alpha=g^{\frac12}_\beta$ on $U_\alpha\cap U_\beta$, and it is equivalent to
$$\left|e^{C_\alpha}\right|\tilde{g}^{\frac12}_\alpha=\left|e^{C_\beta}\right|\tilde{g}^{\frac12}_\beta.$$
Hence $\left|e^{C_\alpha}\right|=\left|e^{C_\beta}\right|=r$ for some $r>0$. Consequently, we have $g=r^2\tilde{g}$. Hence $g$ is unique up to positive constant multiple.

On the other hand, let $g$ be a singular flat metric on $S$ and $\Sigma$ its singular set. Set $S^o=S\setminus\Sigma$.  Let $\{(U_\alpha,z_\alpha)\}\cup\{(U_k,z_k)\}$ be a Leray atlas adapted to  $(S^o,\Sigma)$. Since $g$ is flat on $S^o$ for each $\alpha$ there exists a harmonic function $u_\alpha$ on $U_\alpha$ such that $g_\alpha^{\frac{1}{2}}=e^{u_\alpha}|dz_\alpha|$ on $U_\alpha$. Furthermore, one has
  \begin{equation}\label{flat1}
    e^{u_\alpha-u_\beta}=|\xi_{\alpha\beta}|
  \end{equation}
  on $U_\alpha\cap U_\beta$, where $\xi_{\alpha\beta}=\frac{dz_\alpha}{dz_\beta}$. Since $U_\alpha$ is simply connected, there exists dual harmonic functions $v_\alpha$ for $u_\alpha$ on $U_\alpha$ such that $f_\alpha=u_\alpha+i v_\alpha$ defines a holomorphic function on $U_\alpha$. Set
  \begin{equation}\label{flattomer}
      \eta_\alpha:=df_\alpha
   \end{equation}
  on $U_\alpha$; we claim that the $\eta_\alpha$'s are representatives of a holomorphic connection $\nabla$ on $S^o$. By \eqref{flat1} we have $u_\alpha-u_\beta=\log|\xi_{\alpha\beta}|.$
  Since $\xi_{\alpha\beta}:U_\alpha\cap U_\beta\to \mathbb{C}$ is a nonzero holomorphic function then $F:=f_\alpha-f_\beta-\log\xi_{\alpha\beta}$
   is a holomorphic function on $U_\alpha\cap U_\beta$ with $\mathrm{Re}\, F=0$. Hence there exists $C\in \mathbb{R}$ such that $F=iC$, i.e.,  $f_\alpha-f_\beta=\log\xi_{\alpha\beta}+iC$.
  After differentiating the last equality we have
  $$\eta_\alpha-\eta_\beta=\frac{d\xi_{\alpha\beta}}{\xi_{\alpha\beta}}$$
  which is \eqref{mer1}. We have defined a holomorphic connection $\nabla$ on $S^o$. It is not difficult to see that $g$ is adapted to $\nabla$ on $S^o$. By Theorem \ref{realperiod} it follows that $\nabla$ has real periods on $S^o$.

Let  $p_k$ be a conical singularity of order  $c_k$ of the flat metric $g$ and  $(U_k,z_k)$  the chart of the atlas centered at $p_k$.  Then by Lemma \ref{l1} there exists a holomorphic function $F_k$ on $U_k$ such that the singular flat metric is $$g^{\frac12}=|z_k|^{c_k}\left|e^{F_k}dz_k\right|$$ on $U_k$.  Set
   $$\eta_k:= \left(\frac{c_k}{z_k}+F_k'\right)dz_k$$
  on $U_k$. We  claim that we can extend $\nabla$ to a meromorphic connection represented by $\eta_k$ on $U_k$. We have to check that this definition satisfies condition \eqref{mer1}. Let $(U_\alpha,z_\alpha)$ be a chart with $U_\alpha\cap U_k\ne \emptyset$.  Then there exists a holomorphic map $F_\alpha: U_\alpha\to \mathbb{C}$ such that $g^{\frac12}=|e^{F_\alpha}dz_\alpha|$.
 Let $V\subseteq U_\alpha\cap U_k$ be a simply connected open set. Then by definition of flat metric we have
    $$|z_k|^{c_k}|e^{F_k}|=|e^{F_\alpha}\xi_{\alpha k}|$$
    on $V$, where $\xi_{\alpha k}:=\frac{dz_\alpha}{dz_k}$. This is equivalent to
    $$c_k\log |z_k|+\mathrm{Re}\, F_k=\mathrm{Re}\,F_\alpha+\log|\xi_{\alpha k}|.$$
   Since $V$ is simply connected it is not difficult to see that there exists a constant $C\in\mathbb{C}$ such that
$$c_k\log z_k+F_\beta-F_k-\log \xi_{\alpha k}=C$$
on $V$. Consequently,
\begin{equation}\label{mer3}
  \eta_k=\eta_\alpha+\frac{d\xi_{\alpha k}}{\xi_{\alpha k}}
\end{equation}
 on $ V$. Since $\eta_k$, $\eta_\alpha$ and $\xi_{\alpha k}$ are well defined on $U_\alpha\cap U_k$ and  the equality \eqref{mer3} holds on any simply connected subset of $U_\alpha\cap U_k$ then \eqref{mer3} holds on $U_\alpha\cap U_k$. Hence we have extended $\nabla$ to a Fuchsian meromorphic connection on $S$. It is easy to see that, if $p_k$ is a conical singularity of order  $c_k$ of the flat metric $g$ then $\mathrm{Res}_{p_k}\nabla=c_k.$

  Let $\tilde{\nabla}$ be another meromorphic connection adapted to $g$. Let $\eta_\alpha$ and $\tilde{\eta}_\alpha$ be the representations of $\nabla$ and $\tilde{\nabla}$ on a chart $(U_\alpha,z_\alpha)$ respectively. Then by \eqref{flattomer} there exists holomorphic primitives $F_\alpha$ and $\tilde{F}_\alpha$ of
$\eta_\alpha$ and $\tilde{\eta}_\alpha$ respectively we have
$$g^{\frac12}=|e^{F_\alpha}dz_\alpha|=|e^{\tilde{F}_\alpha}dz_\alpha|$$
on $U_\alpha$. Hence there exists a constant $C\in\mathbb{C}$ such that $F_\alpha\equiv\tilde{F}_\alpha+C$. Consequently, $\eta_\alpha\equiv \tilde{\eta}_\alpha.$  Hence $\nabla=\tilde{\nabla}$ on $S^o$. Let now $(U_k,z_k)$ be a chart around a pole $p_k$. Let $\eta_k$ and $\tilde{\eta}_k$ be the representations of $\nabla$ and $\tilde{\nabla}$ on $U_k$ respectively. Since $\nabla=\tilde{\nabla}$ on $S^o$ we have
$\eta_k\equiv\tilde{\eta}_k$ on $U_k\setminus\{p_k\}$. Since  $\eta_k$ has a unique extension to $U_k$ we have $\eta_k\equiv\tilde{\eta}_k$ on $U_k$. Thus $\nabla=\tilde{\nabla}$ on $S$. Hence $\nabla$ is the unique meromorphic connection adapted to $g$.

\end{proof}

\begin{corollary} Let  $\nabla$ be a  meromorphic connection on a Riemann surface $S$ with real periods. If $g$ is a singular flat metric adapted to $\nabla$, then a conical singularity of $2\alpha \pi$ of $g$ corresponds to a Fuchsian pole of residue $\alpha-1$ of $\nabla$.
\end{corollary}

\section{Meromorphic $k$-differentials}\label{chapter3}

The main purpose of this section is to describe the relation between meromorphic $k$-differentials and meromorphic connections. We shall show that any meromorphic $k$-differential canonically defines a meromorphic connection. Meromorphic  $k$-differentials are studied by many authors (see for example \cite{BCG,Sch,ST,T}). In this section we recall some results on the theory of meromorphic $k$-differentials.
 \begin{definition}\label{kdiftarifi}
    Let $k\in\mathbb{N}$. Let $\{(U_\alpha,z_\alpha)\}$ be a holomorphic atlas for a Riemann surface $S$. A  \emph{meromorphic $k$-differential} $\phi$ on $S$ is a set of meromorphic function elements $\phi_\alpha$ defined in local charts $(U_\alpha,z_\alpha)$ for which the following transformation law holds:
   \begin{equation}\label{kdifffffffff}
    \phi_\beta(p)=\left(\frac{d z_\alpha}{ d z_\beta}(p)\right)^k\phi_\alpha(p), \ \ p\in U_\alpha\cap U_\beta
   \end{equation}

     \end{definition}
   Globally, a $k$-differential is a global meromorphic section of the line bundle $(T^*S)^{\otimes k}$.

\begin{definition}
  Let $\phi$ be a meromorphic $k$-differential on a Riemann surface $S$ and $p$ a zero (pole) of $\phi$.  The \emph{order} of $p$ is the order of $z_\alpha(p)$ for any  $\phi_\alpha$ representing $\phi$ on a local chart $(U_\alpha,z_\alpha)$ around $p$.
\end{definition}
Indeed, $\frac{dz_\alpha}{dz_\beta}$ is a never vanishing holomorphic function on $U_\alpha\cap U_\beta$, and hence \eqref{kdifffffffff} yields that   $\phi_\alpha$ and $\phi_\beta$ have the same order of zero (respectively, pole) at $p$.

  \begin{definition}
     The \textit{critical points} of a meromorphic $k$-differential $\phi$ on a Riemann surface $S$ are its zeroes and poles. All other points of $S$ are \textit{regular points} of $\phi$.  If $\phi$ has no poles then we say that $\phi$ is a \emph{holomorphic} $k$-differential.   If all points are regular points for $\phi$ then we say that $\phi$ is a \emph{regular} $k$-differential.
    \end{definition}
    \begin{remark}
      If $\phi$ is a $k$-differential on an arbitrary Riemann surface $S$. The restriction of $\phi$ on $S^o$ is a regular $k$-differential where $S^o$ is the complement of the critical set of $\phi$.
    \end{remark}
\begin{proposition}\label{localregular}
  Let $\phi$ be a regular $k$-differential on a Riemann surface $S$. Then there exists an atlas where local coordinates are antiderivatives of $k^{th}$ roots of the $k$-differential, these coordinates are defined up to addition of a constant and multiplication by $e^{\frac{2i\pi}{k}}$. In other words, transition maps are generated by translations and rotations of order $k$.
  \end{proposition}
\begin{proof}
Since $\phi$ is a regular $k$-differential, it is not difficult to see that for any $p\in S$ there exists a chart centered at $p$ such that the representation of $\phi$ is identically equal to one. By the condition \eqref{kdifffffffff} we can see that the transition maps of the atlas consisting of the charts are generated by translations and rotations of order $k$.
\end{proof}

As a consequence, this geometric structure preserves a flat metric.
\begin{proposition}[{see for example \cite{ST}}]\label{qdsfm}
  Let $\phi$ be a meromorphic $k$-differential on a Riemann surface $S$ and $\{(U_\alpha,z_\alpha)\}$ an atlas for $S$.  Then there exists a singular flat metric $g$ on $S$ locally given by
   \begin{equation}\label{adaptedkdtog}
     g^{\frac{1}{2}}=|\phi_\alpha|^{\frac1k}|dz_\alpha|
   \end{equation}
   on $U_\alpha$, where $\phi_\alpha$ is the local representation of $\phi$ on $U_\alpha$.
\end{proposition}

We can see from Proposition \ref{qdsfm} that in the flat metric induced by a $k$-differential, a singularity of order $a>-k$ is a conical singularity of angle $\frac{2\pi(a+k)}{k}$.

\begin{definition}
  Let $\phi$ be a meromorphic $k$-differential on a Riemann surface $S$ and $g$ is the singular flat metric defined as in \eqref{adaptedkdtog}. Set $S^o:=S\setminus\Sigma$, where $\Sigma$ is the set of critical points of $\phi$. A smooth curve $\sigma:[0,\varepsilon)\to S^o$ is a \emph{geodesic} for $\phi$ if it is a geodesic for $g$.
\end{definition}
\begin{definition}
 Let $\phi$ be a meromorphic $k$-differential, $g$ a singular flat metric. We say $\phi$ and $g$ are \emph{adapted} to each other if they have the same geodesics.
\end{definition}

\subsection[Meromorphic connections \& $k$-differentials ]{Relation between meromorphic connections and meromorphic $k$-differentials}\label{kdiffandmerconsec}

\begin{definition}
  Let $\nabla$ be a meromorphic connection and $\phi$  a meromorphic $k$-differential on a Riemann surface $S$.  We say that $\nabla $ and $\phi$ are \emph{adapted} to each other if there exists a singular flat metric $g$ such that $g$ is adapted to $\nabla$ and $\phi$.
\end{definition}

First of all, we describe the relation between regular $k$-differentials and holomorphic connections on a Riemann surface $S$.
\begin{theorem}\label{t4}
  Let $\phi$ be a regular $k$-differential on a Riemann surface $S$. Then there exists a unique holomorphic connection   $\nabla$  on $S$ with monodromy in $\mathbb{Z}_k$  such that $\nabla$ is adapted to $\phi$, where $\mathbb{Z}_k$ is the multiplicative group of $k$-th roots of unity.

  On the other hand, if $\nabla$ is a holomorphic connection with monodromy in $\mathbb{Z}_k$  then there exists a regular $k$-differential $\phi$ adapted to $\nabla$. Moreover,   $k$-differential is unique up to a non-zero constant  multiple.

\end{theorem}
\begin{proof}
  Let $\phi$ be a regular quadratic differential on $S$ and $g$ the  flat metric adapted to $\phi$. Due to Theorem \ref{t8} there exists a unique holomorphic connection $\nabla$ on $S$ such that $g$ is adapted to $\nabla$. Hence $\phi$ is adapted to $\nabla$.  It is also possible to write the relation between the local representations of $\nabla$ and $\phi$. Let $\{(U_\alpha,z_\alpha)\}$ be an atlas for $S$. By Proposition \ref{qdsfm}
we have  $$g^{\frac12}=|\phi_\alpha|^{\frac{1}{k}}|dz_\alpha|$$
   on $U_\alpha$, where $\phi_\alpha$ is the representation of $\phi$ on $U_\alpha$.  Then by \eqref{flattomer} we can see that the local representation of $\nabla$ is given by
  \begin{equation}\label{etavaq}
    \eta_\alpha=\frac{d\phi_\alpha}{k\phi_\alpha}
  \end{equation}
  on $U_\alpha$, where $\phi_\alpha$ is the representation of  $\phi$ on $U_\alpha$.

   Thanks to Proposition \ref{localregular} there exists an atlas $\{(U_\alpha,z_\alpha)\}$  such that the representations of $\phi$ are identically $1$ and the transition functions  are given by
   \begin{equation}\label{transitionfunknabla}
     z_\beta=a_{\alpha\beta} z_\alpha+c_{\alpha\beta}
   \end{equation}
   on $U_\alpha\cap U_\beta$, where $a_{\alpha\beta}\in\mathbb{Z}_k$ and $c_{\alpha\beta}$ is a complex number. By \eqref{etavaq} we can see that the representations $\eta_\alpha$ of $\nabla$ on those charts are identically zero. Hence $\{(U_\alpha,z_\alpha)\}$ is a $\nabla$-atlas. Note that transition functions are given by \eqref{transitionfunknabla}, i.e.,  $\xi_{\alpha\beta}=\frac{dz_\alpha}{dz_\beta}\in\mathbb{Z}_k$.  Consequently, by Lemma \ref{nablaatlas}  $\nabla$ has monodromy in  $\mathbb{Z}_k$.

 On the other hand, let now $\nabla$ be a holomorphic connection with monodromy in $\mathbb{Z}_k$. Then we can find an atlas  $\{(U_\alpha,z_\alpha)\}$  for $S$, holomorphic primitives $K_\alpha$ of $\eta_\alpha$, numbers $\tilde\xi_{\alpha\beta}\in\mathbb{Z}_k$ and constants $c_\alpha\in\mathbb{C}^*$ such that setting $\xi_{\alpha\beta}=\frac{dz_\alpha}{dz_\beta}$ we have
$$
   \frac{e^{K_\alpha}}{e^{K_\beta}} \xi_{\alpha\beta}=\frac{c_\beta}{c_\alpha} \tilde{\xi}_{\alpha\beta}.
$$
 Hence we have
 \begin{equation}\label{mcmpmon} {c^k_\beta}{e^{kK_\beta}}=c^2_\alpha{e^{kK_\alpha}} \xi_{\alpha\beta}^2;\end{equation}
 set
 \begin{equation}\label{con2}
   \phi_\alpha =c_\alpha^ke^{kK_\alpha}
 \end{equation}
 on $U_\alpha$.  Take any charts $U_\alpha$ and $U_\beta$  such that $U_\alpha\cap U_\beta\neq \emptyset$.  Then \eqref{mcmpmon} implies
  $$\phi_\beta=\phi_\alpha\left(\frac{dz_\alpha}{dz_\beta}\right)^k$$
  on $U_\alpha\cap U_\beta$. Hence \eqref{con2} defines a meromorphic $k$-differential $\phi$. It is not difficult to see that $\phi$ is adapted to $\nabla$.

 Let $\phi_1$ and $\phi_2$ be meromorphic $k$-differentials adapted to $\nabla$. Let $g_1$ and $g_2$ be the singular flat metrics adapted to $\phi_1$ and $\phi_2$ respectively. Then $g_1$ and $g_2$ are adapted to $\nabla$.  By Theorem \ref{t8} there exists a constant $c>0$ such that $g_1=c g_2$.
  Hence  $\phi$ is  unique up to a non-zero constant multiple.

\end{proof}
Now we are ready to prove Theorem \ref{kvamermain}.
\begin{proof}[{Proof of Theorem \ref{kvamermain}}]
  Let $\phi$ be a meromorphic $k$-differential on $S$ and $\Sigma$ its singular set. By the previous theorem there exists a holomorphic connection $\nabla$ on $S^o:=S\setminus\Sigma$ which has monodromy in $\mathbb{Z}_k$ such that $\phi$ is adapted to $\nabla$.

Let $\{(U_\alpha,z_\alpha)\}\cup \{(U_h,z_h)\}$ be a Leray atlas adapted to $(S^o,\Sigma)$.  Let $p$ be a critical point of the $k$-differential $\phi$ with order $m\in\mathbb{Z}$ and $(U_h,z_h)$ the chart of the atlas centered at $p$.  Since the order of the critical point $p$ is $m$ up to shrink $U_h$ there exists a never vanishing holomorphic function $H_h:U_h\to \mathbb{C}$ such that
  $$\phi_h=z_h^m H_h$$
   where $\phi_h$ is the representation of $\phi$ on $U_h$.
   Let  $\eta_h$ be  the representation of $\nabla$ on $U_h\setminus\{p\}$. By \eqref{etavaq} we have
  \begin{equation}\label{exten}
    \eta_h=\frac{mz_h^{m-1}H_h+z^m_h H_h'}{kz_h^m H_h}dz_h=\left(\frac{m}{kz_h}+\frac{H_h'}{kH_h} \right)dz_h.
  \end{equation}
  Hence $\eta_h$ has a meromorphic extention to $p$  with the residue $\frac{m}{k}$ at $p$. Let $(U_\alpha,z_\alpha)$ be a chart of $S^o$ such that $U_\alpha\cap U_h\ne \emptyset$. Indeed, the condition \eqref{merrepresentation} holds on $U_\alpha\cap U_h$ because $\nabla$ is a holomorphic connection on $S^o$.  By extending local representations of the holomorphic connection $\nabla$  we can extend $\nabla$ to whole surface $S$ as a meromorphic connection $\tilde{\nabla}$. By \eqref{exten} we can see that resulting meromorphic connection $\tilde{\nabla}$ is Fuchsian with residues in  $\frac{1}{k}\mathbb{Z}$. Since the extension is unique $\tilde{\nabla}$ is a unique meromorphic connection adapted to $\phi$.

Now, suppose $\nabla$ is a Fuchsian meromorphic connection with monodromy in $\mathbb{Z}_k$ and  residues in $\frac{1}{k}\mathbb{Z}$. Let $\Sigma$ be the set of poles of $\nabla$. Set $S^o=S\setminus\Sigma$. By the previous theorem there exists a regular $k$-differential $\phi$ on $S^o$ adapted to $\nabla$.

 Let $\{(U_\alpha,z_\alpha)\}\cup \{(U_h,z_h)\}$ be a Leray atlas adapted to $(S^o,\Sigma)$. Let $p$ be a pole for $\nabla$ and $\mathrm{Res}_p\nabla=\frac{m}{k}$ for some $m\in\mathbb{Z}$. Let $(U_h,z_h)$ be the chart centered at $p$. Then the local representation $\eta_h$ on $U_\alpha$ has the form
$$\eta_h=\left(\frac{m}{kz_h}+f_h\right) dz_h,$$
where $f_h:U_h\to \mathbb{C}$ is a holomorphic function. Let $\phi_h$ be the representation of $\phi$ on $U_h\setminus\{p\}$ and $V_\alpha\subset U_h\setminus\{p\}$ a simply connected open set. By \eqref{con2} we have
$$\phi_h|_{V_\alpha}=e^{m\log z_h+kF_\alpha(z_h)}= z_h^{m}e^{kF_\alpha(z_h)},$$
 where  $F_\alpha$ is a suitable  holomorphic  primitive of $f_h$ on $V_\alpha$. Since $U_h$ is simply connected and $f_h$ is a holomorphic function on $U_h$ we can see that $F_\alpha$ has unique holomorphic extension to $U_h$, i.e, there exists a holomorphic primitive $F_h$ of $f_h$ on $U_h$ such that $F_h|_{V_\alpha}=F_\alpha$. Consequently, $ z_h^{m}e^{kF_\alpha}$ has a unique holomorphic extension to
 $U_h\setminus\{p\}$ and we have
 $$\phi_h= z_h^{m}e^{kF_h}.$$
   It is easy to see that $\phi_h$ can be extended to $U_h$ as a meromorphic (or holomorphic if $m\in\mathbb{N}$) function $\tilde{q}_h$. Let  $(U_\alpha,z_\alpha)$ be a simply connected chart of $S^o$ with $U_\alpha\cap U_h\ne \emptyset$. Then the transformation rule \eqref{kdifffffffff} holds because $\phi$ is a regular $k$-differential on $S^o$ and $\tilde{q}_h$ coincides with the representation $\phi_h$ on $U_h\setminus\{p\}$. Consequently, by extending the local representations of $\phi$ we can extend $\phi$ to $S$.
   Since the extension is unique $\phi$ is unique up to a non-zero constant multiple.
\end{proof}

\section{The canonical cover}\label{canonicalcovering}

Let $\nabla$ be a meromorphic connection on a compact Riemann surface $S$. In this section we introduce the notion of canonical cover  induced by $\nabla$ (see \cite[Section 2]{Lan} for quadratic differentials and \cite[Section 2.1]{BCG}  for $k$-differentials).

\begin{theorem}\label{canonicalcover}
Let $S$ be a compact Riemann surface and $\Sigma$ a finite set. Let $\nabla$ be a holomorphic connection on  $S^o:=S\setminus\Sigma$. Assume that $\nabla$ has monodromy in $G$ and that there exists $k\in\mathbb{N}^*$ so that $\arg G^k=\{0\}$. Then there exists a compact Riemann surface $\hat{S}$, a possibly ramified (at the poles for which the argument of the monodromy coefficient is not trivial) covering $\pi: \hat{S}\to S$, and a holomorphic connection $\hat{\nabla}$ on $\hat{S}^0=\hat{S}\setminus\hat{\Sigma}$, where $\hat{\Sigma}=\pi^{-1}(\Sigma)$, such that
 \begin{enumerate}
   \item a smooth curve $\hat{\sigma}:[0,\varepsilon)\to \hat{S}^o$ is a geodesic for $\hat{\nabla}$ if and only if $ \sigma=\pi\circ\hat{\sigma}$ is a geodesic for $\nabla$;
   \item any self-intersecting geodesic of $\hat\nabla$ is closed.
 \end{enumerate}

 \end{theorem}

 \begin{proof}
 Let $k\in\mathbb{N}^*$ be the minimum integer satisfying $\arg G^k=\{0\}$.  Let $\{(U_\alpha,z_\alpha)\}$ be a Leray $\nabla$-atlas for $S^o$ with the transition functions  $z_\beta= a_{\alpha\beta} z_\alpha+c_{\alpha\beta}$ on $U_\alpha\cap U_\beta$, where $a_{\alpha\beta}\in G$ and $c_{\alpha\beta}\in\mathbb{C}$. Note that the representation of $\nabla$ on $U_\alpha$ is identically zero.

  Let $\zeta$ be a fixed primitive $k$-th root of unity. Since $a_{\alpha\beta}\in G$ there exists $l_{\alpha\beta}\in\{0,...,k-1\}$ such that
  \begin{equation}\label{lalfabetan}
    a_{\alpha\beta}=\zeta^{l_{\alpha\beta}}|a_{\alpha\beta}|.
  \end{equation}

   Take $k$ copies of $U_\alpha$ and denote them by $U_{\alpha_j}\cong U_\alpha$ for $j=0,...,k-1$.  We set $z_{\alpha_j}:= \zeta^jz_\alpha$ on $U_{\alpha_j}$.  Whenever $U_\alpha\cap U_\beta\ne \emptyset$, we glue the copies $U_{\alpha_i}$ and $U_{\beta_j}$ for the indices $$j=i-l_{\alpha\beta}\mod k.$$
  Then we have
    \begin{align*}
      z_{\beta_j}=\zeta^j z_\beta & = \zeta^j (a_{\alpha\beta} z_\alpha+c_{\alpha\beta})\\
       &  =\zeta^{j+l_{\alpha\beta}} |a_{\alpha\beta}|z_\alpha+\zeta^jc_{\alpha\beta}\\
       &=\zeta^{j+l_{\alpha\beta}-i}
        |a_{\alpha\beta}|z_{\alpha_i}+\zeta^jc_{\alpha\beta}\\
       & =|a_{\alpha\beta}|z_{\alpha_i}+\zeta^jc_{\alpha\beta}
    \end{align*}
    Hence
    $$z_{\beta_j}=a_{\alpha_i\beta_j}z_{\alpha_j}+c_{\alpha_i\beta_j}$$ for some $a_{\alpha_i\beta_j}\in\mathbb{R}^+$ and $c_{\alpha_i\beta_j}\in \mathbb{C}$. Consequently, we get  a Riemann surface $\hat{S}_1^o$. Take  a connected component of $\hat{S}_1^o$ and denote it by $\hat{S}^o$. Hence we have a covering map $\pi':\hat{S}^o\to S^o$. By Theorem 8.4 from \cite{F} there exists a compact Riemann surface $\hat{S}$ and a (possibly ramified) covering $\pi:\hat{S}\to S$ such that $\pi|_{\hat{S}^o}=\pi'$ and $\pi^{-1}(\Sigma)$ is a finite set.  Hence we have a (possibly ramified) covering $\pi: \hat{S}\to S$.

   Moreover, by the construction there exists a Leray atlas $\{(V_\gamma,w_\gamma)\}$ on $\hat{S}^o$ with the transition functions  $w_\theta= a_{\gamma\theta} w_\gamma+c_{\gamma\theta}$ on $U_\gamma\cap U_\theta$, where $a_{\gamma\theta}\in \mathbb{R}^+$ and $c_{\gamma\theta}\in\mathbb{C}$. Consequently, by setting $\hat{\eta}_\gamma\equiv0$ local 1-forms on $V_\gamma$ we can get a holomorphic connection $\hat{\nabla}$ on $\hat{S}^o$  with the local representations $\{\hat{\eta}_\gamma\}$.

 Let $\hat{\sigma}:[0,\varepsilon)\to \hat{S}^o$ be a smooth curve and  $\sigma:=\pi\circ\hat{\sigma}$. Let $(V_\gamma,w_\gamma)$ be a chart with $V_\gamma\cap\mathrm{supp}(\hat{\sigma})\ne\emptyset$. Then there exists a chart $(U_\alpha,z_\alpha)$ and an integer $j$ such that $V_\gamma\cong U_\alpha$ and $w_\gamma=\zeta^j z_\alpha$. We can see that $z_\alpha(\sigma(t))$ is a Euclidean segment if and only if $w_\gamma(\hat{\sigma}(t))$ is a Euclidean segment. Hence $\sigma$ is a geodesic for $\nabla$ if and only if $\hat{\sigma}(t)$ is a geodesic for $\hat{\nabla}$.

 Let $\hat{G}$ be the monodromy group of $\hat{\nabla}$ on $\hat{S}^o$. Since the atlas $\{(V_\gamma,w_\gamma)\}$ has the transition functions  $w_\theta= a_{\gamma\theta} w_\gamma+c_{\gamma\theta}$ on $U_\gamma\cap U_\theta$, where $a_{\gamma\theta}\in \mathbb{R}^+$ and $c_{\gamma\theta}\in\mathbb{C}$, we can see that $\arg\hat{G}=\{0\}$. Hence any self-intersecting geodesic of $\hat\nabla$ is closed.

 \end{proof}
 \begin{definition}
   The covering $\pi: \hat{S}\to S$ from the last theorem is said to be \emph{the canonical cover induced by $\nabla$.}
 \end{definition}
 The next lemma shows a relation between the $\omega$-limit sets  of geodesics $\nabla$ and $\hat{\nabla}$.
\begin{lemma}\label{pi(w)=W}
Let $\nabla$ be a meromorphic connection on a compact Riemann surface $S$ and $\Sigma$ the set of poles of $\nabla$. Set $S^o=S\setminus\Sigma$.  Let  $\sigma:[0,\varepsilon)\to S^o$ be a maximal geodesic of $\nabla$. Assume $\nabla$ has monodromy in $G$ with $\arg G^k=\{0\}$ for some $k\in \mathbb{N}^*$. Let $\pi:\hat{S}\to S$ be a canonical cover of $S$ induced by $\nabla$. Let $\hat{\sigma}$ be a lift of $\sigma$. Let $\hat{W}$ be the $\omega$-limit set of $\hat{\sigma}$. Then $\pi(\hat{W})$ is the $\omega$-limit set of ${\sigma}$.
\end{lemma}
\begin{proof}
  Let $W$ be the $\omega$-limit set of $\sigma$. It is not difficult to see that $\pi(\hat{W})\subseteq W$.  Let $p\in W$. Then there exists a sequence of positive numbers $t_n\uparrow \varepsilon$ such that $\lim\limits_{n\to\infty}\sigma(t_n)=p$. Take the sequence $\{\hat{\sigma}(t_n)\}$.  Since $\hat{S}$ is compact  $\{\hat{\sigma}(t_n)\}$ has  a limit set $A\subset W$. It is enough to show that there exists $\hat{p}\in A$ such that $\pi(\hat{p})=p$. Since $A$ is not empty there exists a point $\hat{q}\in A$. Let $\pi(\hat{q})=q$ for some $q\in S$. Then there exists a subsequence $\{\hat{\sigma}(t_{n_l})\}_{l=1}^\infty$ such that $\lim\limits_{l\to\infty}\sigma(t_{n_l})=q$. Since the sequence $\{{\sigma}(t_n)\}$ has single limit point we can see that $q=p$. We are done.

\end{proof}
By using Lemma \ref{pi(w)=W} and Theorem \ref{canonicalcover} we prove our second main result.

\begin{proof}[{Proof Theorem \ref{abateconiK}}]
 Since $\sigma$ is a self-intersecting geodesic  the $\omega$-limit set of ${\sigma}$ can not be a single pole.  Let $\pi:\hat{S}\to S$ be the canonical cover induced by $\nabla$ and $\hat{\nabla}=\pi^*\nabla$. Let $\hat{\sigma}$ be a geodesic for $\hat{\nabla}$ such that $\pi\circ\hat{\sigma}=\sigma$. Thanks to Theorem \ref{canonicalcover} we can see that $\hat{\sigma}$ does not intersect itself. Consequently, by using Theorem \ref{t3} we have a list of possible $\omega$-limit sets of the corresponding lift geodesic $\hat{\sigma}$. Let $\hat{W}$ be the $\omega$-limit set of $\hat{\sigma}$. Then Lemma \ref{pi(w)=W} implies $\pi(\hat{W})=W$.

 Let first $\hat{W}$ be the support of a closed geodesic. Since $\pi$ sends a geodesic of $\hat{\nabla}$ to a geodesic of $\nabla$ we have  $W$ the support of a (possibly non-simple) closed geodesic.

 Let now $\hat{W}$ be a graph of saddle connections. By the same argument as above we can see that $W$ is a graph of (possibly self-intersecting) saddle connections.

 Finally, if $\hat{W}$ has non-empty interior then $W$ has also non-empty interior. Assume  $W$ has non-empty boundary. Then each component of the boundary of $\hat{W}$ is a graph of saddle connections with no spikes and at least one pole.  Since $\pi$ is a ramified cover, each component of $\partial W$ is a graph of geodesics (in general  $\partial W$ may not be a graph of saddle connections). Since each component of the boundary of $\hat{W}$ is a graph of saddle connections  every geodesic segment of  $\partial W$ is a part of a saddle connections. We are done.

\end{proof}

\end{document}